\documentclass[letterpaper, 10 pt, conference]{ieeeconf}  

\IEEEoverridecommandlockouts                              
\title{\LARGE \bf
Designing Optimal Personalized Incentive\\ for Traffic Routing using BIG Hype algorithm
}

	\author{Panagiotis D. Grontas,  Carlo Cenedese, Marta Fochesato,\\ Giuseppe Belgioioso,  John Lygeros, Florian D\"orfler
		\thanks{This  work  was supported by NCCR Automation, a National Centre of
Competence in Research, funded by the Swiss National Science Foundation (grant number $180545$).}
		\thanks{Authors are with Automatic Control Laboratory, Department of Electrical Engineering and Information Technology,
        ETH Z\"urich, Physikstrasse 3 8092 Z\"urich, Switzerland. (e-mail:{\tt \{pgrontas, ccenedese, mfochesato, gbelgioioso, jlygeros, dorfler\}@ethz.ch}).
        }
		}

\usepackage[utf8]{inputenc}
\usepackage[T1]{fontenc}
\usepackage{graphicx}
\graphicspath{ {./Images/} }
\usepackage{xcolor}
\usepackage{grffile}
\usepackage{longtable}
\usepackage{wrapfig}
\usepackage{rotating}
\usepackage[normalem]{ulem}
\usepackage{amsmath}
\usepackage{amssymb}
\usepackage{textcomp}
\usepackage{amssymb}
\usepackage{capt-of}
\usepackage{caption}
\usepackage{subcaption}
\usepackage{hyperref}
\usepackage{listings}
\usepackage{float}
\usepackage{optidef}
\usepackage{cite}
\newcommand{\setR}{\mathbb{R}}
\newcommand{\setN}{\mathbb{N}}

\newcommand{\mc}[1]{\mathcal{#1}}
\newcommand{\tup}[1]{\textup{#1}}

\newcommand{\bolds}[1]{\boldsymbol{#1}}
\newcommand{\abs}[1]{\left| #1 \right|}
\newcommand{\norm}[1]{\lVert #1 \rVert}
\newcommand{\proj}[1]{\mathbb{P}_{#1}}

\newcommand{\jac}{\boldsymbol{\mathrm{J}}}

\newcommand{\conserv}{\boldsymbol{\mathcal{J}}}
\DeclareMathOperator{\diag}{diag}

\DeclareMathOperator{\ncone}{N}

\DeclareMathOperator{\sol}{SOL}

\newcommand{\agents}{\mathcal{N}}
\newcommand{\cy}{\bolds{y}}
\newcommand{\cys}{ \cy^{\star} }
\newcommand{\csigma}{\bolds{\sigma}}

\newtheorem{definition}{Definition}

\newtheorem{lemma}{Lemma}

\newtheorem{proposition}{Proposition}

\newtheorem{remark}{Remark}

\usepackage[acronym]{glossaries}
\newacronym{SoC}{SoC}{State of Charge}
\newacronym{PEV}{PEV}{Plug-in Electric Vehicle}
\newacronym{EV}{EV}{Electric Vehicle}
\newacronym{FV}{FV}{fuel vehicle}
\newacronym{TA}{TA}{Traffic Authority}
\newacronym{VRP}{VRP}{Vehicle Routing Problem} 
\newacronym{TTT}{TTT}{Total Travel Time} 
\newacronym{NE}{NE}{Nash equilibrium}
\newacronym{PG}{PG}{Pseudo-Gradient}
\newacronym{BIG Hype}{BIG Hype}{Best Intervention in Games using Hypergradients}
\newacronym{VI}{VI}{Variational Inequality}

\newcounter{algorithm}
\newenvironment{algorithm}[1][]
{	\refstepcounter{algorithm}
	\begin{minipage}{\linewidth}
		\medskip
		\hrule
		\smallskip
		\textsc{Algorithm \thealgorithm}. #1
		\smallskip
		\hrule 
		\smallskip
	\end{minipage}
}
{
	\smallskip
	\hrule width\linewidth\relax
	\smallskip
}

\newcommand{\nodes}{\mc{V}}
\newcommand{\roads}{\mc{E}}

\newcommand{\charge}{\mc{C}^{\textup{c}}}
\newcommand{\park}{\mc{C}^{\textup{p}}}
\begin{document}

\maketitle
\thispagestyle{empty}
\pagestyle{empty}


\begin{abstract}
    We study the problem of optimally routing plug-in electric and conventional fuel vehicles on a city level. In our model, commuters selfishly aim to minimize a local cost that combines travel time, from a fixed origin to a desired destination, and the monetary cost of using city facilities, parking or service stations. The traffic authority can influence the commuters' preferred routing choice by means of personalized discounts on parking tickets and on the energy price at service stations. We formalize the problem of designing these monetary incentives optimally as a large-scale bilevel game, where constraints arise at both levels due to the finite capacities of city facilities and incentives budget. Then, we develop an efficient decentralized solution scheme with convergence guarantees based on \textit{BIG Hype}, a recently-proposed hypergradient-based algorithm for hierarchical games. Finally, we validate our model via numerical simulations over the Anaheim's network, and show that the proposed approach produces sensible results in terms of traffic decongestion and it is able to solve in minutes problems with more than 48000 variables and 110000 constraints. 
\end{abstract}
\section{Introduction}
The problem of optimally manage urban cities is of paramount important in modern society. In fact, the EU alone incurs an annual cost of more than $267$ billion euros \cite{eu:2019:traffic} due to high level of congestion in major cities.
Besides affecting citizens' daily life,  inefficiencies in the  transportation system are responsible for broad  environmental damages, since higher levels of traffic congestion lead to higher CO2 emissions \cite{barth:2009:traffic}. While increasing road capacity or building alternative routes is a traditional approach to managing traffic demand and does offer effective benefits, policymakers and researchers are nowadays exploring other interventions taking advantage of the cities facilities and optimizing the flows of vehicles that daily travel within it by imposing tolling or designing incentives. 

Toward this goal, a popular concept is \textit{congestion pricing} schemes that dates back to \cite{beckmann1956studies} and that propose to \textit{toll} heavily congested roads to influence commuters routing choices, i.e., the associated \gls{VRP}, with the overall objective of decreasing traffic congestion. During the years, a large body of work has been carried out on this topic, often through the lens of game theory showing the potential that taxation (or incentives) have in changing commuters behaviour. Most of the works study the effect that a given pricing scheme produces in terms of network decongestion, see \cite{de2005congestion,de2011traffic} and references therein. 
In \cite{mahnoosh:2017:optimal_pricing_routing_game}, the authors use routing games to analyze the effect that a fleet of \glspl{PEV} has on both the traffic and electricity network showing the importance of considering facilities while studying \gls{VRP} on urban networks. Routing games can also be used to model the effect of 
 \glspl{PEV} in performing load balancing in smart grids \cite{etesami2020smart}.  

Arguably of higher interest is the computation of the optimal set of incentives (or taxes) that must be designed to achieve a beneficial effect in terms of decongestion. In \cite{beckmann1956studies}, the authors use a marginal congestion cost to design tolls, while in \cite{paccagnan2021optimal}, the authors set of tolls that minimizes the system inefficiency. 

These classical results are difficult to generalize to the case in which the \gls{TA} has to meet some limitation such as a desired budget or localized interventions, viz. act only over a limited number of roads/facilities in the network. A natural extension to these more complex setups can be obtained via bi-level games (or optimization problems). If the \gls{TA} aims at designing a set of tolls this problem takes the name of \textit{restricted network
tolling problem} and is NP-hard. The problem is naturally ill-posed and thus the authors rely on heuristic small network sizes \cite{verhoef2002second} or solutions that lack theoretical guarantees of optimality \cite{lawphongpanich2004mpec}. In \cite{paccagnan:2022:data_driven_congestion_pricing} the authors propose a data-driven approach based on the scenario approach to design robust tolls. 
In 
\cite{sohet2021hierarchical,sohet2020coupled}, the authors propose a multi-level optimization problem to compute optimal incentives focusing mostly on \glspl{PEV}. In particular, the presence of \glspl{PEV} allows to design two sets of incentives, one provided by the \gls{TA} and the  other associated to the vehicles smart-charging. Clearly, the design of the two has to be performed holistically. Also in this case the proposed algorithm lacks optimality guarantees. In \cite{decentralized_joint_routing_planning,understanding_the_impact} the authors study the impact of service stations and parking lots on the flows of commuters and show that design incentives for such facilities can be a viable way to influence (indirectly) the \gls{VRP}. 

In summary, a great body of work has been produced in designing tolls to influence the \gls{VRP} that commuters undergo everyday. Yet, due to the inherent problem complexity and lack of scalable algorithms, it is still not clear how to design incentives in the case of limited budget resources and/or of targeted interventions. Moreover, the study on how discounts on facilities influence the \gls{VRP} has been limited. In this work, we aim at bridging both these gaps by exploiting \gls{BIG Hype} (recently introduced in \cite{grontas:2023:big_hype}). This allows us not only to maintain the problem complexity and constraints over the maximum interventions, but also to guarantee (local) optimality of the set of interventions proposed.

The main contributions of this paper can be summarized as follows:
\begin{itemize}
    \item We design a bi-level game that describes the \gls{VRP} in an urban area where commuters have to leave their vehicles at predefined facilities, viz. parking lots and charging stations. The \gls{TA} is able to influence the commuters' routing by providing personalized discounts to access these facilities.
    \item Both levels of the game are subject to constraints, in particular we model both the capacity of the facilities and impose constraints over the maximum discounts that the \gls{TA} can provide. Moreover, we also ensure that the designed policy does not exceed the predefined maximum budget that the \gls{TA} has allocated.
    \item We design an iterative gradient based algorithm that provably converges to the optimal set of personalized incentives, while ensuring that the commuters choice of routing is socially stable, i.e., it is a \gls{NE}.
    \item The cornerstone of the algorithm is \gls{BIG Hype}, thus our algorithm can be implemented in a distributed and highly scalable manner, without requiring common simplifying assumptions, e.g., that the decision-making process undergone by the commuters can be described via a potential game.
    \item We propose numerical studies showing that the proposed algorithm scales gracefully also for urban networks of big dimensions.
\end{itemize}

\section{Problem formulation}
In this section, we formalize the \gls{VRP} in an urban area where a \gls{TA} has the possibility to affect the routing choice of a subset of commuters by designing the  price for accessing some facilities around the city. Namely, it can provide discounts for specific parking lots, and  for the electricity price at charging stations.  Naturally, the former influences those commuters owning either a \glspl{FV} or 
 a \gls{PEV} while the latter only \gls{PEV} owners. We consider facilities in which commuters leave their car for several hours during the day, e.g., the parking lots used during working hours. We assume that the goal of the \gls{TA} is to alleviate traffic congestion, nevertheless the proposed formulation allows to easily change this objective for others such as maximizing the revenues for the facilities, see Remark~\ref{rem:diff_cost_leader}.  The commuters' goal is to arrive at their desired destination via their \gls{PEV} or \gls{FV}. Among them, a subset is reactive to the proposed discounts, and thus take part to the \gls{VRP} since they select their route  minimizing their overall cost. The rest of the population can be modelled as an exogenous demand creating both a constant amount of traffic congestion and occupying always the same facilities independently from their price.


We model the transportation network as a (strongly) connected directed
graph \( \mc{G} \coloneqq (\nodes, \roads) \),
where the nodes \( \nodes \) represent intersections or point of interest, while the edges
\( \roads \subseteq \nodes \times \nodes \) correspond to the roads connecting them.
Given an edge \( \varepsilon \in \roads  \), we denote \( \varepsilon_-, \varepsilon_+ \in \nodes\)
its starting and ending node, respectively. 
The charging stations and parking lots can be located at one of the nodes of the network and are described by the sets  \( \charge \subseteq \nodes \) and $ \park \subseteq \nodes $,  respectively. Is common that at the same location there are both service stations and normal parking, therefore in general   $\charge\cap\park\neq\emptyset $.  
We let $\mathcal C = \charge \cup \park$ \( n_e \coloneqq \abs{\roads} \), \( n_v \coloneqq \abs{\nodes} \), \( n_{\tup  c} \coloneqq \abs{\charge} \), and \( n_{\tup p} \coloneqq \abs{\park} \). 

\subsection{Lower Level: Routing and Charging/ Parking Game}
We group the commuters into \( N \) classes of similar characteristics, 
referred to as \textit{agents} and indexed by the set \( \agents \coloneqq \{1, \ldots, N\} \).
Each agent \( i \in \agents \) represents a population of \glspl{PEV} or \glspl{FV}   composed of $P_i$ vehicles sharing the same origin and destination nodes, denoted by  \( (o_i, d_i) \in \nodes\times \nodes \) where we assume $o_i\neq d_i$. Each driver aims at reaching $d_i$, but has to leave the car at a node $j\in\mc C$ and then complete its trip by travelling from $j$ to $d_i$. Notice that we allow for $d_i\notin \mc C$. 
Each agent $i$ seeks to determine the fraction of vehicles that will park at the lot $\ell\in\park$, denoted by \( g_i^{\tup p,\ell} \in [0, 1] \). If $i$ is composed of \glspl{PEV}, the fraction that will charge at each station \( j \in \charge \),
is denoted by \( g_i^{\tup c,j} \in [0, 1] \), while if the vehicles in $i$ are \glspl{FV} then \( g_i^{\tup c,j}=0\) for all $j\in \mc C^{\tup c}$. For conciseness, we let 
\( g_i^{\tup c} \coloneqq (g_i^{\tup c,j})_{j \in \nodes}\in\setR^{\abs{\nodes}}\), where $g_i^{\tup c,j}=0$ if $j\notin \charge$, similarly we define \(g_i^{\tup p} \coloneqq (g_i^{\tup p,\ell})_{\ell \in \nodes}\in\setR^{\abs{\nodes}} \), and   the collective vectors
 \( \bolds{g}^\tup{c} \coloneqq (g_i^\tup{c})_{i \in \agents},~ \bolds{g}^\tup{p} \coloneqq (g_i^\tup{p})_{i \in \agents} \).

Further, agent \( i \) determines the route towards a charging station or parking lot
via the decision variable \( \phi_i^\varepsilon \in [0, 1] \), that indicates the percentage of vehicles
traversing road \( \varepsilon \in \roads \), so \(\phi_i\coloneqq (\phi_i^\varepsilon)_{\varepsilon\in\mc E},~\bolds{\phi} \coloneqq (\phi_i)_{i \in \agents} \).
To ensure compatibility of vehicle flows we impose, for all \( v \in \nodes \), the constraint:
\begin{equation} \label{eq:flow_compatibility}
	\sum_{ \mathclap{\varepsilon:\varepsilon_+=v} } \phi_i^{\varepsilon} \; - \; \sum_{ \mathclap{\varepsilon:\varepsilon_-=v} } \phi_i^{\varepsilon} =
	\begin{cases}
		-1, ~ & v = o_i, \\
		g_i^{\tup{c},j}+g_i^{\tup{p},j}, ~ &   \text{otherwise}
	\end{cases}.
\end{equation}

The goal of agent \( i \in \agents \) is to choose \( g_i^{\tup{c}} \), \( g_i^{\tup{p}} \) and \( \phi_i \) so as to minimize
its cost function, namely \( f_i \), consisting of three terms: travel time, charging/parking cost,
and last-mile cost.

The travel time of agent \( i \in \agents \) is given by
\begin{equation}
	f_i^{\textrm{t}} \coloneqq \eta_i \sum_{\varepsilon \in \roads} P_i \phi_i^\varepsilon t_{\varepsilon} (\sigma_{\varepsilon}(\bolds{\phi})),
\end{equation}
where \( \eta_i \) is the (monetary) value of time, and
\( \sigma_{\varepsilon}(\boldsymbol{\phi}) \coloneqq \sum_{i \in \agents} P_i \phi_i^{\varepsilon} \) is the aggregate agents'
flow on road \( \varepsilon \).
We denote \( t_{\varepsilon}(\cdot) \) the latency on \( \varepsilon \) as a function of total agents' flow. Similarly to \cite{decentralized_joint_routing_planning,understanding_the_impact}, 
we derive our latency function form the one used by the Bureau of Public Roads and attain the following affine function
\begin{equation}
\label{eq:travel_time_edge}
	t_{\varepsilon}(\sigma_{\varepsilon}(\bolds{\phi})) \coloneqq a_{\varepsilon} + b_{\varepsilon} (h_{\varepsilon} + \sigma_{\varepsilon}(\bolds{\phi})),
\end{equation}
where \( a_{\varepsilon}, b_{\varepsilon} \) are positive constants (see \cite{decentralized_joint_routing_planning}, \cite{understanding_the_impact} for details),
and \( h_{\varepsilon} \) represents the  flow due to vehicles that are not reactive to the discounts or traffic conditions and it is assumed fixed.

The charging cost reads as:
\begin{equation}
\label{eq:cost_chargign}
	f_i^{\text{c}} \coloneqq \sum_{j \in \charge} q_i g_i^{\tup{c},j} (\overline c_j^{\tup{c}} - c_j^{\tup{c},i}),
\end{equation}
where \( q_i \) is the total amount of electricity that all the  \glspl{PEV} in class \( i \)  should purchase to fully charge their batteries,
\( \overline c_j^{\tup{c}}>0 \) is the base price of electricity at station \( j \in \charge \),
whereas \( c_j^{\tup{c},i} \geq 0\) is the discount provided for class \( i \) in station \( j \) by the \gls{TA}. We stress that \( c_j^{\tup{c},i} \) is a design variable of the \gls{TA}. Among the \glspl{PEV} in each population $i\in\mc N$, there can be a small percentage 
$\overline g_i^{\tup{c}}\in[0,1]$ that necessitates to charge during the day due to an initial low state of charge, this translates into the local constraint
\begin{equation}
\label{eq:percentage_need_charge}
  \mathbf{1}^\top g_i^{\tup{c}} \geq \overline g_i^{\tup{c}}.
\end{equation}

Similarly, the cost of parking reads as 
\begin{equation}
\label{eq:cost_parking}
	f_i^{\text{p}}  \coloneqq \sum_{j \in \park}  g_i^{\tup{p},j} (\overline c_j^{\tup{p}} - c_j^{\tup{p},i}),
\end{equation}
where the cost of parking at $j\in\park$ and the associated discount for agent $i$ are denoted by $\overline c_j^{\tup{p}},~ c_j^{\tup{p},i}$, respectively. We denote the vector of all the discounts provided by $\bolds c\coloneqq(c^{i})_{i\in\agents}\in\setR^m$ and \( m \coloneqq N (n_{\tup c}+ n_{\tup p})  \), while $c^{i}\coloneqq(c_j^{i})_{j\in\mc C}$, and
$c_j^{i}\coloneqq(c_j^{\tup{c},i},c_j^{\tup{p},i})$; we define $\bolds {\overline c}$ similarly. 

Finally,  each agent \( i \in \agents \) encounters a cost for travelling from the   charging station/ parking $j\in\mc C$ (where they left their car) to their destination $d_i$, the last-mile cost. To model this, let us define a vector $\hat g_i\in\setR^{|\nodes|}$ as a vector of all $0$s except for its $d_i$-th component that is equal to $1$. 
Then, the last-mile cost is defined as
\begin{equation}
	f_i^{\text{lm}} \coloneqq \eta_i \norm{  g_i^{\tup{p}}+g_i^{\tup{c}} - \hat{g}_i}^2_W,
\end{equation}
where \( W \succ 0 \) is a diagonal matrix with positive elements. The above cost  is zero if all the vehicles in $i$ leave their vehicle at $d_i$. 
The weights in $W_{jj}$ represents the discomfort that agent $i$ faces for travelling to $d_i$ from the facility $j\in \mc C$. There are several ways to model $W$ and depend not only on the particular city's structure, but also on the different means of transportation that can be used to cover the last-mile trip. Hereafter, we assume that the discomfort is proportional to the time required to move from $j$ to $d_i$ in free-flow conditions via the shortest path. 
Notice that from \eqref{eq:travel_time_edge}, it follows that  the time necessary to travel through road $\varepsilon\in\mc E$ in free-flow conditions is $a_\varepsilon>0$. Moreover, we impose $W_{d_id_i}=\epsilon>0$ where $\epsilon$ is a small scalar that ensures $W\succ0$.  

To guarantee that each agent has the possibility to access the selected facility, we introduce a constraint over the maximum number of vehicles of class $i$ that can use $j\in\mc C$. 
For every $j\in\charge$, we denote the facility maximum number of slots allocated a priori for agent $i$ by the \gls{TA} by $\delta_i^{\tup{c},j}>0$, thus the agents decisions must satisfy
\begin{equation}
    \label{eq:capacity_charge_follow}
     P_ig_i^{\tup{c},j} \leq \delta_i^{\tup{c},j}.
\end{equation}
Analogously, we define the same constraint for the parking facilities as follows 
\begin{equation}
    \label{eq:capacity_park_follow}
     P_ig_i^{\tup{p},j} \leq \delta_i^{\tup{p},j}.  
\end{equation}

We denote the decision variable of each agent as \( y_i \coloneqq (\phi_i, g_i^{\tup c},g_i^{\tup p}) \in \setR^{n_i} \),
constrained to the local feasible set  
\begin{equation}
    \mc{Y}_i \coloneqq \{ y_i \in [0, 1]^{n_i} \, | \, \eqref{eq:flow_compatibility},\,\eqref{eq:percentage_need_charge},\,\eqref{eq:capacity_charge_follow},\,\eqref{eq:capacity_park_follow} \text{ hold} \} ,
\end{equation} 
where \( n_i = n_e + 2n_v \).
We let \( \bolds{y} \coloneqq ((y_i)_{i \in \agents}) \in \mc{Y} \) be the collective strategy profile of all agents,
where \( \mc{Y} \coloneqq \prod_{i \in \agents} \mc{Y}_i \subseteq \setR^n \) and \( n \coloneqq \sum_{i \in \agents} n_i \).
Agent \( i \in \agents \) is faced with the optimization problem
\begin{equation} \label{eq:lower_level_single}
	\begin{alignedat}{2}
		&\underset{\displaystyle y_i \in \mc{Y}_i}{\mathclap{\mathrm{minimize}}} 
		\quad~ && \quad f_i(\bolds c, y_i, \bolds{\sigma}(\cy)), 
	\end{alignedat} \tag{\( \mc{P}_i \)}
\end{equation}
where \( f_i \coloneqq f_i^{\textrm{t}}  + f_i^{\text{c}} + f_i^{\text{p}} + f_i^{\text{lm}}  \) is the sum of  the objective functions, and 
  \( \bolds{\sigma}(\cy) \coloneqq ((\sigma_\varepsilon)_{\varepsilon \in \roads}) \) is the aggregative agents' flow over $\mc G$.
Observe that \( f_i \) depends both on the lower-level variables \( y_i \) and the aggregative quantity \( \bolds{\sigma}(\cy) \),
owned by the agents \( i\in\agents \). Moreover, it also depends on the upper-level variable \( \bolds c \), controlled by the \gls{TA}, thus \eqref{eq:lower_level_single} is an aggregative parametric game.

\subsection{Upper Level: Personalized Incentive Design}
Next, we describe the objective of the \gls{TA} that is assumed here to be interested into minimizing the traffic congestion over the network or on relevant parts of it.
As anticipated, it can design a set of personalized  incentives, viz. $\bolds c$ , for each agent $i$ and facility $j\in\mc C$. Therefore, the the \gls{TA} aims at minimizing the \gls{TTT} that reads as 
\begin{equation}
\label{eq:g_c}
	\varphi(\bolds c, \csigma(\cy)) \coloneqq 
	\sum_{\varepsilon \in \roads_D} (h_\varepsilon + \sigma_{\varepsilon}(\bolds{\phi}))
	(a_\varepsilon + b_\varepsilon (h_\varepsilon + \sigma_{\varepsilon}(\bolds{\phi})))
\end{equation}
where \( \roads_D \subseteq \roads \) corresponds to a subset of the road network
that the \gls{TA} is interested in decongesting, e.g., the city center.
For example since $2007$ the city of Z\"urich uses a perimeter control scheme to limit the congestion in the city center while allowing for higher levels of congestion in other areas \cite{ambuhl:2018:zurich_mfd}.
Naturally, if the \gls{TA} aims at decongesting the entire network it would set \( \roads_D = \roads \).
It is important to highlight that, the dependence of \( \varphi \) on \( \bolds c \) is implicit, i.e., the discounts \( \bolds c \) shape the traffic pattern \( \bolds{\phi} \) which, 
in turn, determines the \gls{TTT} over $\roads_D$.

\begin{remark}[Different cost functions]\label{rem:diff_cost_leader}
    The same problem formulation described above can be easily modified to address other scenarios in which the \gls{TA} is replaced by a private entity, e.g., the owner of the facilities, that aims at maximizing the revenue attained. In such a case, \eqref{eq:g_c} should be simply replaced by 
    $\tilde \varphi = - \sum_{i\in\mc N} (f^{\tup c}_i+f^{\tup p}_i)$.
    In this case, the goal of the \gls{TA} is to minimize the provided discount while maximizing the commuters that use the most costly facilities. The iterative scheme based on \gls{BIG Hype} discussed in Section~\ref{sec:big_hype} can be used to compute the optimal set of discounts for this alternative formulation with minor adjustments.
\end{remark}

To avoid large discrepancies in facility prices for different agents,
we restrict the personalized discounts to the feasible set
\( \mc{D} \coloneqq \{ \bolds c \in \setR^{m} \, | \,  c^i_j \in [0,\, \theta^i_j\overline c^i_j] \} \),
where \( \theta^i_j \in (0,1)\). 
Further, we assume that the \gls{TA} has to compensate the facilities for the provided discounts, 
but can only spend a limited budget \( C > 0 \).
Concretely, the \gls{TA}'s budget constraint reads
\begin{equation} \label{eq:budget_constraint}
	C^{\tup b}(\bolds c,\bolds y)\coloneqq \sum_{i \in \agents} \left(\sum_{j \in \charge} q_i g^{\tup c,j}_i c^{\tup c,j}_i+ \sum_{\ell \in \park} P_ig^{\tup c,\ell}_i c^{\tup c,\ell}_i \right) \leq C.
\end{equation}
For computational reasons, we model \eqref{eq:budget_constraint} as a soft constraint via
the penalty function
\begin{equation}
	\varphi^{\text{b}}(\bolds c,\cy) \coloneqq \max \left\{ C^{\tup b}(\bolds c,\bolds y)  - C, 0 \right\}^2.
\end{equation}
Hence, the \gls{TA}'s objective is $$ \varphi_{\text{TA}}(\bolds c,\cy) \coloneqq \varphi(\bolds c,\csigma) + \mu \varphi^{\text{b}}(\bolds c,\cy) ,$$
where \( \mu > 0 \) is a penalty parameter used to ensure that the final \gls{TA}'s discounts require satisfy \eqref{eq:budget_constraint}.  

\subsection{Bilevel Game Formulation}
Overall, our considered traffic model operates as follows.
The \gls{TA} broadcasts the personalized incentives \( \bolds c \) to the commuters in $\agents$.
Then, the latter respond by choosing the  routing and facilities 
that solve the collection of interdependent optimization problem \( \mc{L}(\bolds c) \coloneqq \{ \text{\eqref{eq:lower_level_single}} \}_{i \in \agents} \)
which constitute a \textit{parametric Nash game}, with parameter $\bolds c$.

A relevant solution concept for \( \mc{L}(\bolds c) \) is the Nash equilibrium that corresponds to a strategy
profile where no agent can reduce its cost by unilaterally deviating from it.
\begin{definition}
	Given \( \bolds c \), a strategy profile \( \cys \in \mc{Y} \) is a \gls{NE}
	of \( \mc{L}(\bolds c) \) if, for
	all \( i \in \agents \):
	\begin{equation*}
		f_i(\bolds c , y_i^{\star}, \csigma(\cys)) \leq f_i(\bolds c, y_i, \csigma((y_i,\cys_{-i}))), ~ \forall y_i \in \mc{Y}_i. \text{\tag*{ $\square$ }}
	\end{equation*}
\end{definition}

In our setting, \( \cys \) is a \gls{NE} if and only if it is a solution to a specific \gls{VI} problem \cite[Prop.\ 1.4.2]{facchinei2003finite}, 
namely, if it satisfies:
\begin{equation} \label{eq:ne_is_vi}
	F(\bolds c, \csigma(\cys))^{\top} (\cy - \cys) \geq 0, \quad \forall \cy \in \mc{Y},
\end{equation}
where \( F(\bolds c, \csigma(\cy)) \coloneqq (\nabla_{y_i} f_i(\bolds c, y_i, \bolds \sigma(\cy)))_{i \in \agents} \) is the so-called \gls{PG} mapping.
We denote \( \sol(F(\bolds c, \cdot), \mc{Y}) \) the set of solutions to \( \mc{L}(\bolds c) \).
Next, we prove existence and uniqueness of a \gls{NE}. 

\begin{lemma} \label{lemma:existence_and_uniqueness}
	For any fixed \( \bolds c \in \mc{D} \), the parametric game \( \mc{L}(\bolds c) \) admits a unique \gls{NE}.
	{\hfill \( \square \)}
\end{lemma}
\begin{proof}
	See Appendix \ref{proof:existence_and_uniqueness}.
\end{proof}

In view of \autoref{lemma:existence_and_uniqueness}, we can define the single-valued
parameter-to-NE mapping \( \cys(\cdot) : \bolds c\mapsto \sol(F(\bolds c, \cdot), \mc{Y}) \),
and use to it express the \gls{TA}'s problem as follows
\begin{subequations} \label{eq:upper_level_substituted}
	\begin{alignat}{2}
		&\underset{\displaystyle x}{\mathclap{\mathrm{minimize}}} 
		\quad~ && \quad \varphi_{\text{TA}}(\bolds c, \cys(\bolds c)) =: \hat{\varphi}_{\text{TA}}(\bolds c), \\
		& \overset{\hphantom{\displaystyle x}}{\mathclap{\mathrm{subject~to}}} \quad~
		&& \quad \bolds c\in \mc{D}.
	\end{alignat}
\end{subequations}
The dependence of \( \varphi_{\text{TA}} \) on \( \cys(\cdot) \) highlights that the \gls{TA} anticipates the rational response of the \glspl{PEV} to \( \bolds c \).

Unfortunately, the implicit nature of \( \cys(\cdot) \) renders \( \hat{\varphi}_{\text{TA}} \) non-smooth and non-convex \cite[Section IV]{grontas:2023:big_hype}.
We can find a globally optimal solution to \eqref{eq:upper_level_substituted} by recasting it as mixed-integer program and using
off-the-self software to solve it, as in \cite{incentive_stackelberg}.
The resulting computational complexity, however, drastically increases with problem size \cite[Section V-B.3]{grontas:2023:big_hype} making them not suitable to solve the problem at hand.
For this reason, in the next section we adopt \gls{BIG Hype} that focuses on obtaining local solutions in a very efficient manner.


\addtolength{\textheight}{-3cm}   

\section{Optimal discount computation via \gls{BIG Hype}}\label{sec:big_hype}
\begin{figure}[tb]
	\centering
	\includegraphics[trim={200 50 200 50},clip,width=\linewidth]{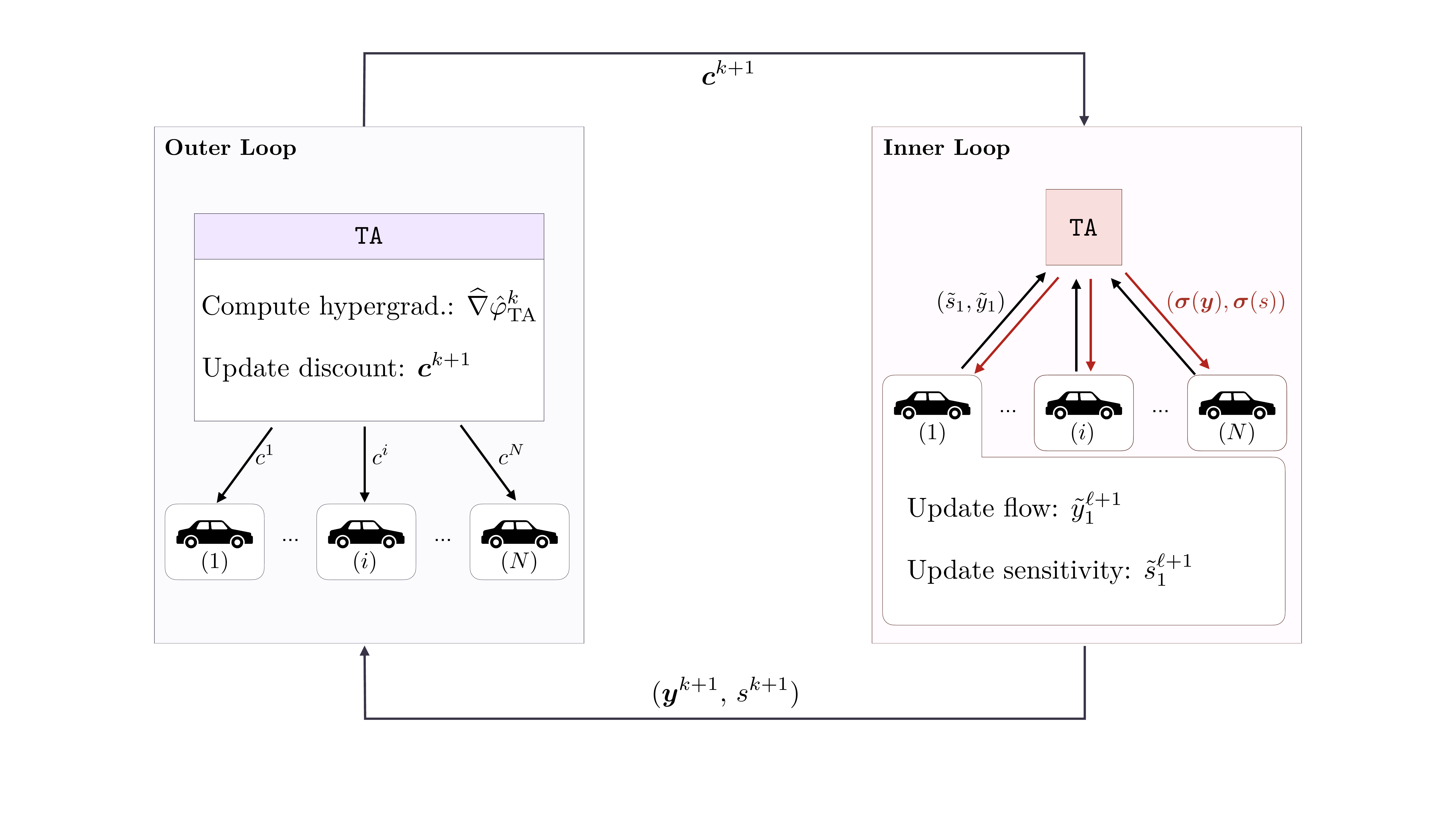}
	\caption{Schematic representation of the outer and inner loop in \autoref{alg:outer_loop} used to compute the optimal discount $\bolds c^\star$. During every outer iteration $k$ the inner loop is iterated until $ \max \big\{\|\tilde{\cy}^{\ell}\! -\! \tilde{\cy}^{\ell-1} \|, \|\tilde{s}^{\ell}\! -\! \tilde{s}^{\ell-1} \| \big\} \leq \sigma$. }
	\label{fig:traffic_graph}
\end{figure}
Bilevel games like \eqref{eq:upper_level_substituted} are notoriously difficult to solve due to the pathological lack of smoothness and convexity.
Further, we are interested in scalable solution methods that can exploit the inherent hierarchical and distributed nature of the game, such as \cite{grontas:2023:big_hype} and \cite{local_stack_seeking}.

In this work, we employ the novel \gls{BIG Hype} algorithm that has been firstly developed in \cite{grontas:2023:big_hype} for a wider class of bilevel games. It requires  weak assumptions on the upper level  objective, i.e., the \gls{TA}'s objective $\varphi_{\text{TA}}$ can be designed in a general form,
and it utilizes simple update rules that can be easily implemented. Moreover, it preserve the distributed structure of the agents' problem allowing to design  a scheme that is efficient also for a great number of agents, see Section~\ref{sec:sims}.
In its core, \gls{BIG Hype} uses projected gradient descent to obtain a local solution of \eqref{eq:upper_level_substituted}.
Informally, the gradient of \( \hat{\varphi}_{\text{TA}} \), commonly referred to as the hypergradient, can be characterized using the chain rule as follows
\begin{equation*}
	\nabla \hat{\varphi}_{\text{TA}}(\bolds c) = \nabla_{\bolds c} \varphi_{\text{TA}}(\bolds c, \cys(\bolds c)) + \jac \cys(\bolds c)^{\top} \nabla_{\cy} \varphi_{\text{TA}}(\bolds c, \cys(\bolds c)).
\end{equation*}
To compute \( \nabla \hat{\varphi}_{\text{TA}}(\bolds c) \), the \gls{TA} requires knowledge of \( \cys(\bolds c) \)
as well as its Jacobian \( \jac \cys(\bolds c) \), which is known as the \textit{sensitivity} and,
intuitively, represents how the commuters react  to a marginal change in the discounts $\bolds c$.
\begin{remark}
    Technically, \gls{BIG Hype} computes the so-called \textit{conservative gradient} of \( \hat{\varphi}_{\text{TA}} \),
    denoted by \( \conserv \hat{\varphi}_{\text{TA}} \),
    which is a generalization of the gradient for non-smooth and non-convex functions \cite{conservative_definition}.
\end{remark}

The proposed hierarchical traffic-shaping scheme, derived from \gls{BIG Hype}, is presented in \autoref{alg:outer_loop}, and consists 
of two nested loops, that we describe in the following using also the aid of the scheme in Figure~\ref{fig:traffic_graph}.
The \gls{TA} sends the current personalized discount \( c^i \) to each agent $i$ computed during the outer loop iteration $k$.
Then, during each iteration $\ell$ of the inner loop, see \autoref{alg:inner_loop},
the commuters estimate their routing and facility choice \( \tilde{y}_i^{\ell+1} \in \setR^{n_i} \) along with their
sensitivity \( \tilde{s}_i^{\ell+1} \in \setR^{n_i \times m} \) using the aggregative quantities $\csigma(\tilde{\cy}^{\ell})$, $\csigma(\tilde{s}^{\ell})$ broadcast at the end of each iteration $\ell$ by the \gls{TA}. 
The inner loop terminates once the estimates are sufficiently accurate without, however, requiring an exact evaluation
of \( \cys(\bolds c) \) and \( \jac \cys(\bolds c) \).
Then, in the outer loop, the \gls{TA} gathers the approximate \gls{NE} and its sensitivity, \( \cy^{k+1} \) and \( s^{k+1} \), 
which are used to update \( \bolds c\) via a projected hypergradient step, where $\widehat \nabla$ denotes the attained inexact hypergradient.
\begin{remark}[Agents' decision update]
	To update the flow profile as in \autoref{alg:inner_loop}, the agents need to project onto the polyhedron \( \mc{Y}_i \).
	This corresponds to repeatedly solving a parametric quadratic program, which can be performed efficiently
	using appropriate solvers, such as \cite{osqp}.
	On the other hand, the TA's feasible set is a box and, hence, the projection onto it can be performed analytically.
\end{remark}
\begin{remark}[Agents' sensitivity update]
	The sensitivity updates step of \autoref{alg:inner_loop} requires that each agent $i$ computes the auxiliary matrices \( S_{1,i} \),  \( S_{2,i} \), and \( S_{3,i} \),
	that store the partial Jacobian of the mapping \( \proj{\mc{Y}_i}[y_i - \gamma F(\bolds c , \csigma(\cy))]  \) with respect to $\bolds c$, $y_i$, and $\csigma$, respectively. 
	This computation is non-trivial as it requires differentiating through the projection operator \( \proj{\mc{Y}_i} \).
	A theoretical and numerical study of this problem was presented in \cite{optnet_arxiv}. Moreover, the aggregative structure of the agents' game \eqref{eq:lower_level_single} makes the sensitivity $s_i$ dependent only on the aggregate sensitivity of the other agents $\bolds \sigma(s)$. This is an important feature of  \autoref{alg:inner_loop}, since the decisions and sensitivity of the single agents is not communicated to the other commuters but is only available to the \gls{TA}, maintaining the local agents' preference private. 
\end{remark}

\begin{figure} [t]
\begin{algorithm}[Hierarchical Traffic Shaping]%
	\label{alg:outer_loop}%
	\textbf{Parameters}:
	Step sizes \( \{ \alpha^k \}_{k \in \setN} \), tolerances \( \{\sigma^k\}_{k \in \setN} \).\\[.2em]
	\textbf{Initialization}: $k\leftarrow 0$, 
	\( {\bolds c}^k \in \mc{D},~ \cy^k \in \setR^n, ~ s^k \in \setR^{m \times n} \).
	\\[.2em]
	\textbf{Iterate until convergence:}\\
	$
	\left \lfloor
	\begin{array}{l}
		\text{\texttt{(TA)}}\\
        [.2em]
		\left|\begin{array}{l}
  \text{Compute inexact hypergradient:} \\[.1em]
		\begin{array}{l}
			\widehat{\nabla} \hat \varphi_{\text{TA}}^k =
			\nabla_1 \varphi_{\text{TA}}(\bolds c^k, \cy^{k}) + (s^k)^{\top} \nabla_2 \varphi_{\text{TA}}({\bolds c}^k, \cy^{k}) \\		
		\end{array}
		 \\ [.45em]
           \text{Update discounts:} \\[.1em]
		\begin{array}{l}
			\bolds c^{k+1} = \proj{\mathcal{D}}[\bolds c^{k} - \alpha^k \widehat{\nabla} \hat\varphi_{\text{TA}}^k] \\
		\end{array}\\[.45em]
        \text{Send \( \bolds c^{k+1}\) to \texttt{(PEVs/FVs)}}
        \end{array}
        \right.\vspace{.45em}\\
		\text{\texttt{(TA+PEVs/FVs)}}
        \\[.2em]
		\left|\begin{array}{l}
        \text{Estimate flows and sensitivity:}\\
		\begin{array}{l}
			(\cy^{k+1}, \ s^{k+1}) = \textbf{Inner Loop} ({\bolds c}^{k+1},\, \cy^k,\, s^k,\, \sigma^k) 
			 \end{array}
    \end{array}
		\right.\vspace{.5em}\\
			k \leftarrow k+1
		\end{array}
		\right.
		$
	\end{algorithm}
\end{figure}
\begin{figure}
	\begin{algorithm}[Inner Loop]%
		\label{alg:inner_loop}%
		\textbf{Parameters}: step size \( \gamma \). \\
		\textbf{Input}:
		\(  \bolds c, \cy, s , \sigma\). \\
		\textbf{Define}: \( h_i(\bolds c,y_i,\csigma(\cy)) \coloneqq \proj{\mc{Y}_i}[y_i - \gamma F(\bolds c , \csigma(\cy))] \) \\
		\textbf{Initialization}:  \(\ell \leftarrow 0,\  \tilde \csigma(\cy^\ell) = \csigma(\cy),\ \csigma(\tilde s^\ell) = \csigma(s),  \) \\
		\( ~~~~~\zeta=0,\, (\forall i \in \agents)\,S_{1,i} = \jac_1 h_i(\bolds c, \tilde y_i^{\ell},\csigma(\tilde\cy^{\ell})),\)\\ 
        \(~~~~~  S_{2,i} =  \jac_2 h_i(\bolds c, \tilde y_i^{\ell},\csigma(\tilde\cy^{\ell})),~  S_{3,i} =  \jac_3 h_i(\bolds c, \tilde y_i^{\ell},\csigma(\tilde\cy^{\ell}))  \) \\
		\textbf{Iterate} \\[0.5em]
		$
		\begin{array}{l}
			\left\lfloor
			\begin{array}{l}
				\text{\texttt{(PEVs/FVs)}  \(\forall i \in \agents \) (in parallel)} \\ 
				\left\lfloor
				\begin{array}{l}
					\text{Update flow profile:} \\
					\begin{array}{l}
						\tilde{y}_i^{\ell+1} = h_i(\bolds c, \tilde y_i^{\ell},\csigma(\tilde\cy^{\ell}))
					\end{array}
					\\[.5em]
					\text{If \(\zeta=1 \):}  \\
                    \left|
					\begin{array}{l}
						S_{1,i} = \jac_1 h_i(\bolds c, \tilde y_i^{\ell},\csigma(\tilde\cy^{\ell})),\\
                        S_{2,i} =  \jac_2 h_i(\bolds c, \tilde y_i^{\ell},\csigma(\tilde\cy^{\ell})),\\
                        S_{3,i} =  \jac_3 h_i(\bolds c, \tilde y_i^{\ell},\csigma(\tilde\cy^{\ell})) 
					\end{array} 
                    \right.\\
					\text{Else:} \\
					\left|
					\begin{array}{l}
						\text{Do not update \( S_{1,i}, S_{2,i}, S_{3,i} \)}
					\end{array}	
                    \right.\\[.5em]
                    \text{Update sensitivity:} \\
					\begin{array}{l}
						\tilde{s}_i^{\ell+1} = 				
						S_{2, i} \tilde{s}_i^\ell+S_{3, i} \csigma(\tilde{s}^\ell) + S_{1, i}\\
					\end{array}
					\vspace*{.5em}
				\end{array}
				\right. \\[1.5em]
                \text{\texttt{(TA) }} \\ 
				\left\lfloor
				\begin{array}{l}
					\text{Gather: $ \tilde{\cy}^{\ell + 1} \coloneqq (y_i^{\ell + 1})_{i \in \agents} ,\, \tilde{s}^{\ell + 1} \coloneqq (\tilde s_i^{\ell + 1})_{i \in \agents} $ } \\[.5em]
					\text{If \( \norm{\tilde{\cy}^{\ell + 1} - \tilde{\cy}^\ell} \geq \sigma \): $\zeta =1$}  \\
					\text{Else: $\zeta =0$}\\[.5em]
                    \text{Broadcast: \( \bolds\sigma(\tilde{\cy}^{\ell + 1}),\, \bolds\sigma(\tilde{s}^{\ell + 1}),\, \zeta \)} 
					\vspace*{.5em}
				\end{array}
				\right. \\[1.5em]
				\ell \leftarrow \ell + 1  \\
				\text{Until } \max \big\{\|\tilde{\cy}^{\ell}\! -\! \tilde{\cy}^{\ell-1} \|, \|\tilde{s}^{\ell}\! -\! \tilde{s}^{\ell-1} \| \big\} \leq \sigma
				%
			\end{array}
			\right. 
		\end{array}	
		$\\[.2em]
		\textbf{Output}: \( \bar \cy= \tilde{\cy}^{\ell}, \bar s= \tilde{s}^{\ell} \).
	\end{algorithm}
\end{figure}

Next, we present the main result of the paper that establishes the convergence of \autoref{alg:outer_loop} to \textit{ a critical point}\footnotemark{} of \eqref{eq:upper_level_substituted}
under appropriate choices of the step sizes \( \{ \alpha^k \}_{k \in \setN}, \gamma \) and tolerance sequence \( \{ \sigma^k \}_{k \in \setN} \).
\footnotetext{Any point \( \bolds c\in \mc{D} \) that satisfies \( 0 \in \conserv \hat{\varphi}_{\text{TA}}(\bolds c) + \ncone_{\mc{D}}(\bolds c) \) is called
a \textit{ critical point} of \eqref{eq:upper_level_substituted},
where \( \ncone_{\mc{D}} \) denotes the normal cone of \( \mc{D} \).
}
\begin{proposition} \label{prop:big_hype_convergence}
	Let \( \{ \alpha^k \}_{k \in \setN} \) be non-negative, non-summable and square-summable,
	let \( \{ \sigma^k \}_{k \in \setN} \) be non-negative and satisfy \( \sum_{k=0}^{\infty} \alpha^k \sigma^k < \infty \),
	and let \( \gamma \) be sufficiently small.
	Then, any limit point of the sequence \( \{ {\bolds c}^k \}_{k \in \setN} \) is a composite critical point of \eqref{eq:upper_level_substituted}.
\end{proposition}
\begin{proof}
	See Appendix \ref{proof:big_hype_convergence}.
\end{proof}

Any local minimum of \( \hat{\varphi}_{\text{TA}} \) is a composite critical point \cite[Prop.\ 1]{conservative_definition}.
However, the set of composite critical points can also include spurious points, e.g., saddles or local maxima,
but in our numerical experience these are rarely encountered.
This concludes the convergence analysis showing that employing the iterative algorithm in \autoref{alg:outer_loop} and \autoref{alg:inner_loop} the \gls{TA} is able to compute a (locally) optimal set of discounts $\bolds c^\star$, that minimizes the \gls{TTT} by influencing the commuters routing.
\section{Numerical Simulations}\label{sec:sims}
\subsection{Simulation Setup}
We deploy our proposed traffic-shaping scheme on the Anaheim city dataset \cite{transportation_datasets}.
For demonstration purposes, we only consider a sub-network of Anaheim consisting of 50 nodes and 118 edges
presented in \autoref{fig:graph}, 
which includes some additional edges to ensure path connectivity of the sub-network.
\begin{figure}[t]
	\centering
	\includegraphics[width=\linewidth]{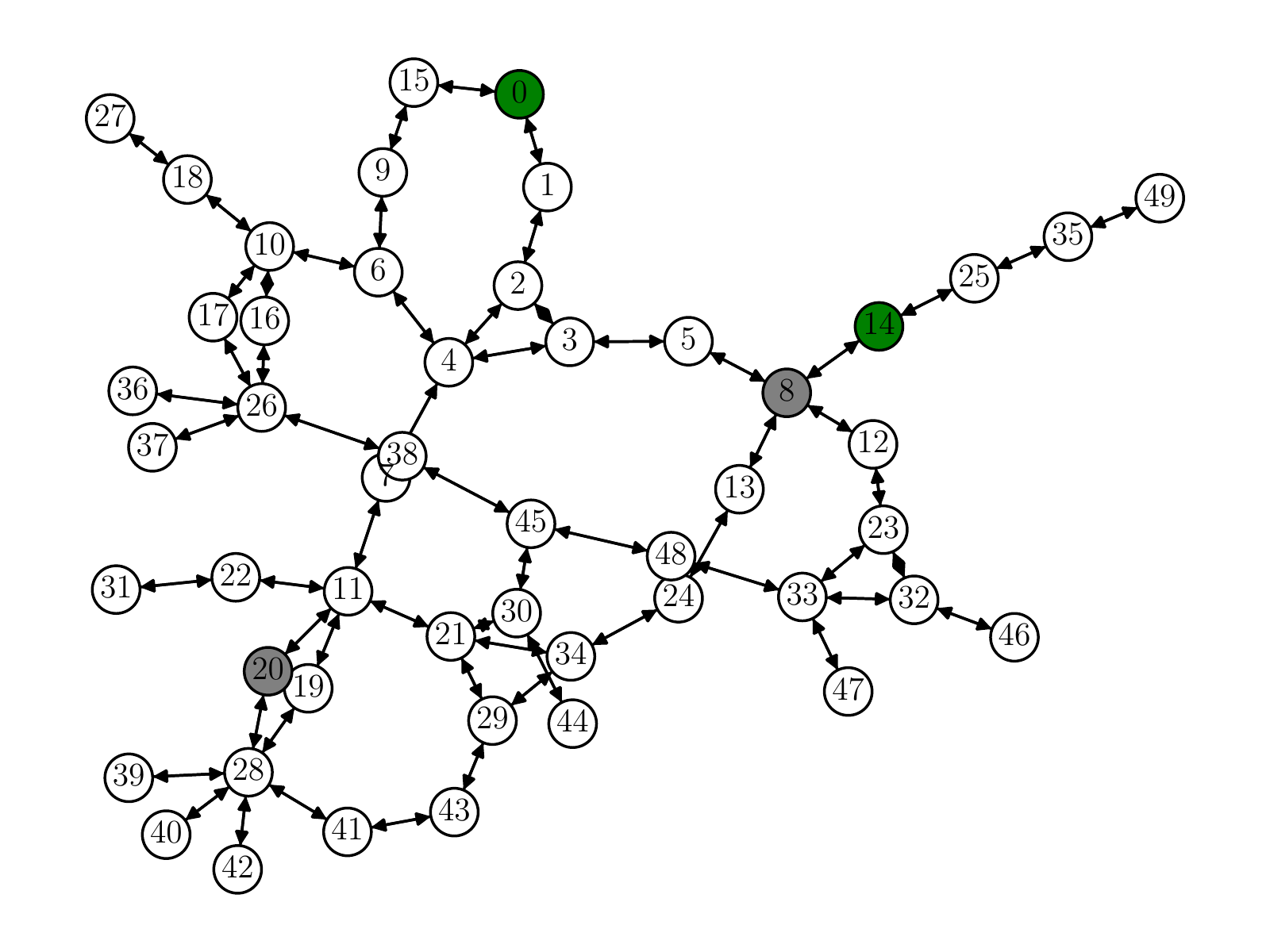}
	\caption{Graph of Anaheim's sub-network where charging and parking stations are represented as green and grey nodes, respectively.}
	\label{fig:graph}
\end{figure}

The parameters \( a_{\varepsilon}, b_{\varepsilon}, h_{\varepsilon} \) are derived from Anaheim's original non-linear
data via linearization as in \cite{understanding_the_impact}.
We let \( \eta_i =  30\,[\$/\text{hr}]\) for all agents, 
while \( q_i \) is drawn uniformly at random in the interval \( [20, 60] \).
We assume the presence of two charging stations located at nodes $0$ and $14$, hence $\charge=\{0,14\}$. The associated prices are set to $\overline{c}_0^{\tup{c}}=0.35\,[\$/\text{kWh}]$  and $\overline{c}_{14}^{\tup{c}}=0.3\,[\$/\text{kWh}]$. Similarly, 
parking are located in nodes $8$ and $20$, hence $\park=\{8,20\}$, and the prices are $\overline{c}_{8}^{\tup{p}}=17\,\$$ and $\overline{c}_{20}^{\tup{p}}=20\,\$$.
The provided discounts are restricted to \( c_i^{\tup c,j} \in [0, 0.2] \), for all \( i \in \agents, j \in \charge\),
whereas \( c_i^{\tup p,j} \in [0, 5] \), for all \( i \in \agents, j \in \park \). The access rate of class $i$ to facility $j$ is set to $\delta_i^{c,j} = 0.75 \frac{P_i}{n_{\tup c}} $ and $\delta_i^{p,j} = 0.75 \frac{P_i}{n_{\tup p}}$ for all $i \in \mathcal{N}, j \in \mc C$.

We consider $20$ classes of \gls{PEV}s and $20$ classes of \gls{FV}s that amount
to $5\%$ and $15\%$ of the total vehicles for their respective type, ensuring a policy penetration rate of $20\%$. Given the total vehicles $n_{\text{veh}} = 242584$ on the network, each  class size is computed as $P_i = \rho_i\frac{n_{\text{veh}}}{20}$, where $\rho_i = 0.05$ if $i$ is composed of \glspl{PEV} and $\rho_i = 0.20$ otherwise. For each \glspl{PEV} class, the minimum fraction of vehicles that need to charge, i.e., $\bar{g}_i^c$, is randomly selected.

\subsection{Parametric Budget}
We investigate the \gls{TTT} reduction attained by our algorithm as a function
of the available budget.
We consider both personalized and uniform discounts, 
where the latter correspond to providing the same discount to all agents that access a particular facility, i.e., $\mc D$ is endowed for every $j\in\mc C$ with the additional constraints $c_i^{\tup c,j}=c_\ell^{\tup c,j}$ and $c_i^{\tup p,j}=c_\ell^{\tup p,j}$ for all $i,\ell\in\agents$.
In \autoref{fig:budget}, we present the \gls{TTT} reduction as a percentage of the \gls{TTT}
without intervention.
\begin{figure}
	\centering
	\includegraphics[width=0.8 \linewidth]{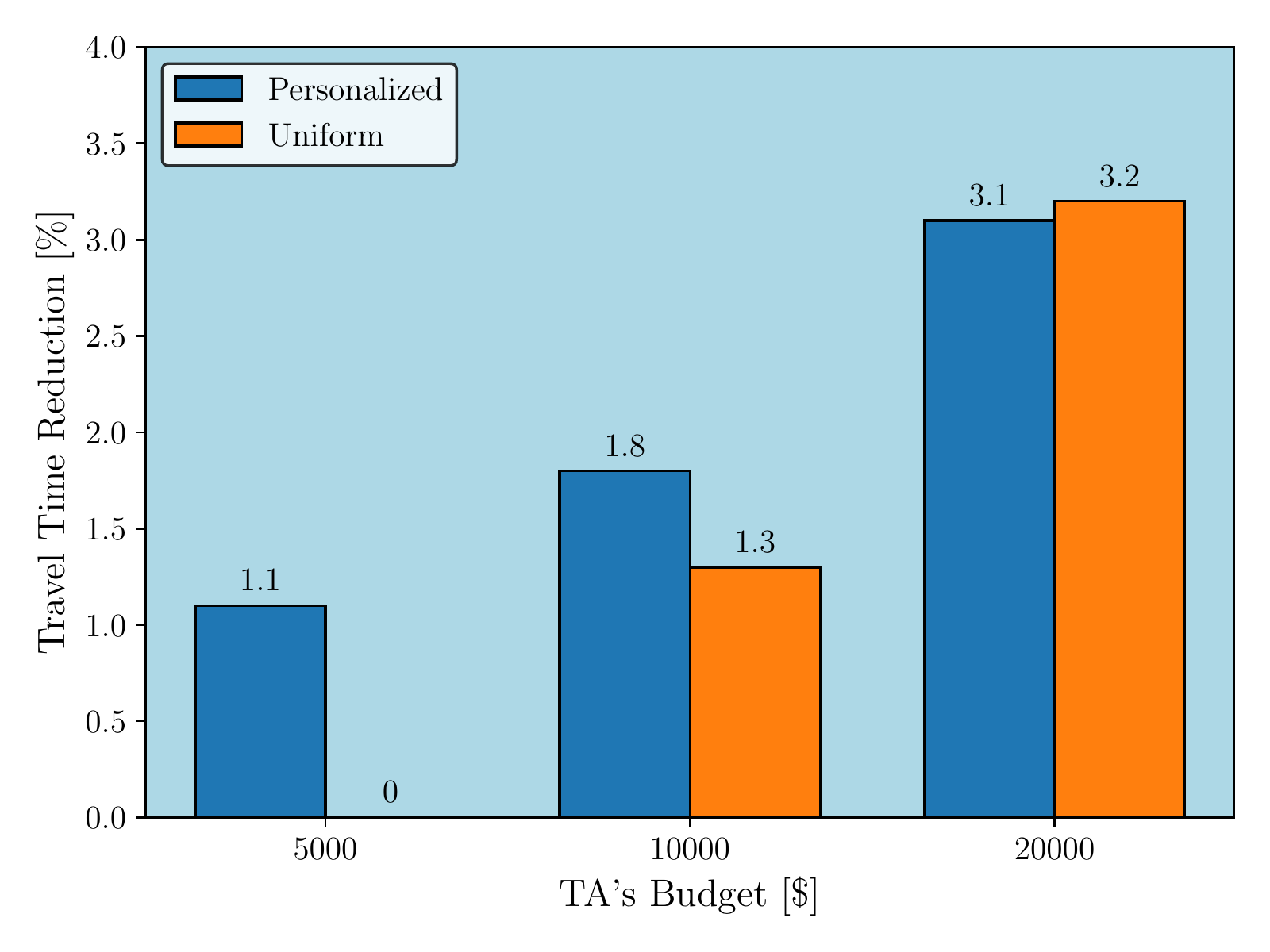}
	\caption{Percentage of \gls{TTT} reduction for different budgets, using personalized and uniform discounts.}
	\label{fig:budget}
\end{figure}
We observe that for large budgets both personalized and uniform incentives perform similarly,
and are able to provide sensible decongestion of the network, around $3\%$. The magnitude of the result is inline with those obtain in other works considering incentives. For example, in an experiment performed in Lee county almost $17\text{M}\,\$$ has been put in place to achieve a traffic reduction of around $5\%$ see \cite{burris:1998:variable_pricing}.  

In our case, as the available budget decreases, personalized incentives outperform uniform
ones due the more efficient and targeted allocation of resources. As an example to showcase the effect of the \gls{TA}'s discounting policy on the routing game, we report in \autoref{fig:flow_diff} the difference in the flow of \textit{controllable} vehicles, i.e., \gls{PEV} and \gls{FV}, with and without \gls{TA} intervention, considering a budget of 5000~\$. In this specific setting, the \gls{TA} tends to offer discounts in the facility placed at node 0 to convey more vehicles towards that node and decongesting areas around remaining facilities, resulting in a decrement of the \gls{TTT} of $76\,\text{h}$ every day.

\begin{figure}
	\centering
	\includegraphics[trim={0cm 0 0cm 0},clip,width=\linewidth]{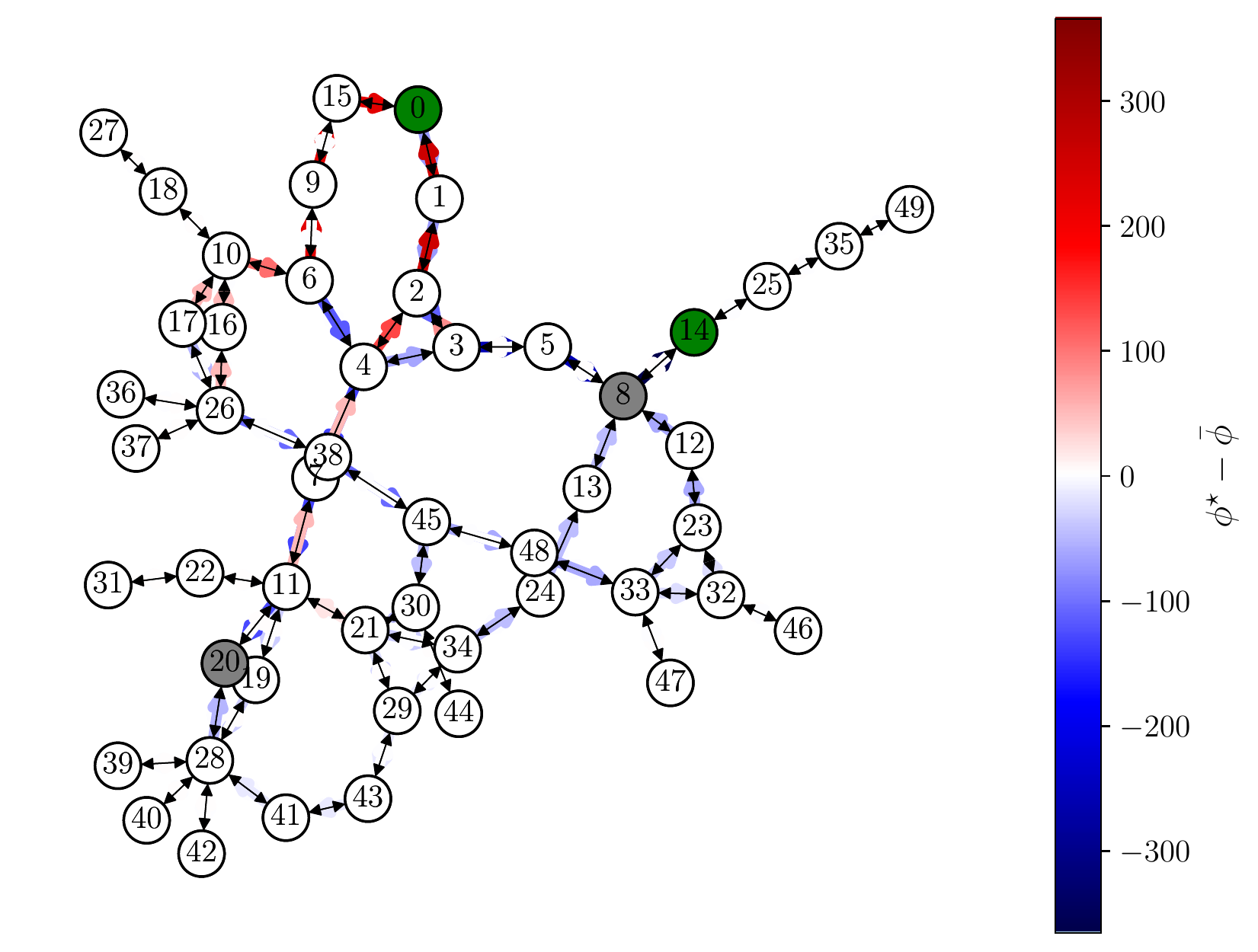}
	\caption{Difference in the flow of \gls{PEV} and \gls{FV} with \gls{TA} intervention (i.e., \gls{BIG Hype}) and without \gls{TA} intervention (i.e., \gls{NE} under no discounts).}
	\label{fig:flow_diff}
\end{figure}

\subsection{Scalability}
Next, we explore the scalability of our proposed scheme by considering
sub-networks of Anaheim of increasing size, starting form $n_{{v}}=50$ to $n_{{v}}=400$.
We consider the computational cost of the \gls{TA}'s updates, in the outer loop,
and the agents' updates, in the inner loop.
Specifically, assuming a distributed implementation, the inner loop cost corresponds to
the maximum computation time among all the agents.
In \autoref{fig:scalability}, we present both computational costs as a function of the number of nodes in the network.
\begin{figure}
	\centering
	\includegraphics[width=1 \linewidth]{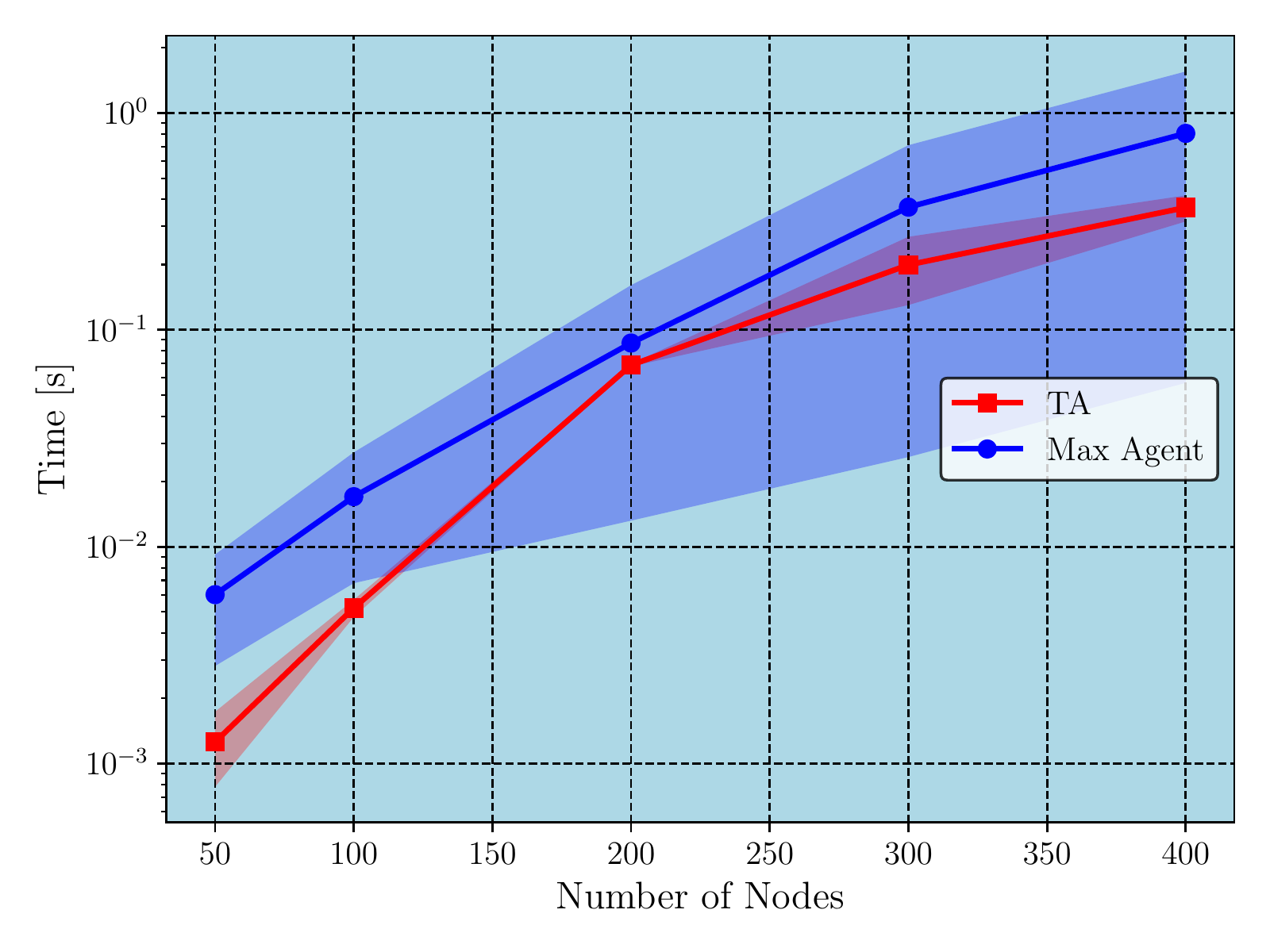}
	\caption{Computation time for inner and outer loop for increasing network size.
					The solid lines represent the average and the shaded area represents \( \pm 1 \) standard deviation over 100 iterations.}
	\label{fig:scalability}
\end{figure}

To asses the overall computation time, we highlight that
the number of outer loop iterations required for convergence is typically 
in the order of hundreds.
Moreover, we can employ large tolerances for the inner loop iteration as in \cite[Section V]{grontas:2023:big_hype}, 
which allow running few inner iterations (usually only one) for each outer loop iteration. 
\section{Conclusion}
Smart incentives on the price of energy at service stations and on the price for accessing parking facilities are a viable option for the \gls{TA} to influence commuters daily routing and promote traffic decongestion. Compared to tolling the effect is limited due to the voluntary nature of such policies and their indirect effect on the \gls{VRP}. To compute the optimal set of discounts \gls{BIG Hype} produces highly scalable solutions that can be applied to networks of great dimension.

The proposed model can be extended in many directions, for example the \gls{TA} can be endowed  with the ability to impose tolls over the networks' roads. This extended model would effectively encompass the restricted network tolling problem making it a compelling, yet highly complex, problem to address. Furthermore, more extensive simulations can be carried out to assess the performance of the proposed algorithm for bigger networks and at the variation of the number of service stations and parking lots.


\bibliographystyle{IEEEtran}
\bibliography{references}

\appendices
\section*{Appendix}
\renewcommand\thesubsection{\Alph{subsection}}
\subsection{Proof of \autoref{lemma:existence_and_uniqueness}}
\label{proof:existence_and_uniqueness}
A sufficient condition for existence and uniqueness of a \gls{NE} is that the \gls{PG} \( F(\bolds c, \cdot) \)
is strongly monotone \cite[Th.\ 2.3.3(b)]{facchinei2003finite}.
Since \( F(\bolds c, \cdot) \) is an affine mapping, strong monotonicity
is equivalent to \( \jac F(\bolds c, \cdot) \) being positive definite \cite[Th.~2.3.2(c)]{facchinei2003finite}.

Observe that the functions \( f_i^{\text{t}} \) and \( f_i^{\text{c}},\, f_i^{\text{p}},\,f_i^{\text{lm}} \) depend only
on \( \bolds{\phi} \) and \( g_i^{\tup c},\,g_i^{\tup p} \), respectively.
Therefore, after applying an appropriate permutation we can express \( \jac F(\bolds c, \cdot) \)
as \( M = \diag(M_{\phi}, M_g) \),
where \( M_{\phi} \coloneqq \jac_{\phi} F(\bolds c, \cdot) \) and \( M_g \coloneqq \jac_{g} F(\bolds c, \cdot) \).
In the proof of \cite[Lem. 1]{decentralized_joint_routing_planning} it is shown \( M_y \succ 0 \),
provided that \( t_{\varepsilon}(\cdot) \) is an affine function.
Moreover, \( M_{g} \succ 0 \) because \( f_i^{\text{c}} + f_i^{\text{p}} + f_i^{\text{lm}} \) is a strongly convex quadratic
that depends only on \( g_i \), for all \( i \in \agents \).
The fact the both \( M_{\phi} \) and \( M_{g} \) are positive definite implies that \( M \succ 0 \),
hence completing the proof.
\( \hfill \square \)

\subsection{Proof of \autoref{prop:big_hype_convergence}}
\label{proof:big_hype_convergence}
To prove the claim, we will verify that Assumption 1 and Standing Assumptions 1\( - \)4 in \cite{grontas:2023:big_hype} are satisfied by our
model and then invoke \cite[Th.~2]{grontas:2023:big_hype}.
Notice that the \( f_i \)'s are convex quadratic in the form of \cite[Eq.~17]{grontas:2023:big_hype}, hence Assumption 1 holds.
Further, Standing Assumption 1 holds true since the feasible sets \( \mc{Y}_i \) are polyhedral.
For Standing Assumption 2, we showed in the proof of \autoref{lemma:existence_and_uniqueness} that \( F(\bolds{c}, \cdot) \) is strongly monotone
for any \( \bolds{c} \).
Uniform strong monotonicity and Lipschitz continuity follows from the fact the \( F \) is affine in \( (\bolds{c}, \cy) \).
Standing Assumption 3 holds since \( F \) is affine which implies that it is semialgebraic and, thus, definable.
Finally, note that \( \varphi_{\text{TA}} \) is semialgebraic, which implies definable, 
and \( \mc{D} \) is convex and compact, hence, verifying Standing Assumption 4.
\( \hfill \square \)

\end{document}